\documentclass[11pt]{article}
\usepackage{fullpage,url,times}
\usepackage{leftidx}
\usepackage[affil-it]{authblk}
\usepackage{kbordermatrix}
\usepackage{amsmath}
\usepackage{amssymb}
\usepackage{amsthm}
\usepackage{cite}
\usepackage{color}
\definecolor{mygray}{gray}{.5}
\definecolor{myblue}{rgb}{0,0.7,0.7}
\definecolor{myred}{rgb}{0,0.5,0}

\renewcommand{\kbldelim}{[}
\renewcommand{\kbrdelim}{]}

\def \K {\mathbb{K}}
\def \F {\mathbb{F}}
\def \E {\mathbb{E}}

\def \N {\mathbb{N}}
\def \Q {\mathbb{Q}}

\def \Z {\mathbb{Z}}
\def \C {\mathbb{C}}

\newcommand{\lset}{\left \{}
\newcommand{\rset}{\right \}}
\newcommand{\set}[1]{\lset #1 \rset}

\newcommand{\si}{\mathop{\rm si}}
\newcommand{\GL}{\mathop{\rm GL}}
\newcommand{\rank}{\mathop{\rm rank}}
\newcommand{\Span}{\mathop{\rm span}}
\newcommand{\Char}{\mathop{\rm char}}

\newcommand{\Der}{\mathop{\rm Der}}

\newcommand{\remove}[1]{}
\newtheorem{lemma}{Lemma}[section]
\newtheorem{ttheorem}[lemma]{Theorem}

\newtheorem{corollary}[lemma]{Corollary}
\newtheorem{claim}[lemma]{Claim}
\theoremstyle{definition}
\newtheorem{definition}[lemma]{Definition}
\newtheorem{example}[lemma]{Example}
\newtheorem{observation}[lemma]{Observation}
\newtheorem{remark}{Remark}

\newenvironment{proof-sketch}{\trivlist\item[]\emph{Proof Sketch}.}

\begin{document}

\title{Secret-Sharing Matroids need not be Algebraic}
\author{Aner Ben-Efraim\thanks{Dept. of Computer Science, Ben Gurion University of the Negev, Be'er Sheva, Israel. Supported by a grant from the Israeli Science and Technology ministry and by Israel Science Foundation grants 938/09 and 544/13. This paper is part of the author's M.Sc. thesis in the Dept. of Mathematics, Ben Gurion University of the Negev, Be'er Sheva, Israel..}}

\maketitle

\begin{abstract}
We combine some known results and techniques with new ones to show that there exists a non-algebraic, multi-linear matroid. This answers an open question by Mat\' u\v s (Discrete Mathematics 1999), and an open question by Pendavingh and van Zwam (Advances in Applied Mathematics 2013). The proof is constructive and the matroid is explicitly given.
\end{abstract}

\section{Introduction}\label{sec:intro}
The aim of this paper is to compare two natural extensions of linear representations of matroids: multi-linear representations and algebraic representations. Specifically, we show that there exists a non-algebraic, multi-linearly representable matroid. 

\paragraph{Linear matroids.}
The beginning of matroid theory dates back to the work of Whitney in 1935. Whitney defined the matroid axioms as an abstraction of the linear dependence properties of a set of vectors. It is therefore not surprising that the class of linear matroids, i.e., matroids that originate from linear dependence, is probably the most well studied class of matroids. 

\remove{A linear representation is a mapping of the matroid's ground set elements into a vector space over some field, such that dependence is maintained. So by definition, a matroid is linear if and only if it is linearly representable over some field.}

The class of linear matroids has many desirable properties, such as closure to minors and duality. 
Quite early in the study of matroids, MacLane proved that not all matroids are linear. The smallest non-linear matroid was presented by V\' amos~\cite{Vam} (the V\' amos matroid). 

Since linear algebra can also be done over division rings, and the linear dependence over division rings clearly satisfies the matroid axioms, the class of matroids representable over division rings is probably the most natural extension of linear matroids. There are examples of non-linear matroids that are representable over division rings\remove{, called non-commutative matroids in \cite{Ing71b}}, the most well-known being the non-Paupus matroid. These matroids were called non-commutative matroids in \cite{Ing71b}.

\paragraph{Characteristic sets.}
Not all matroids are linear, but even a linear matroid is not necessarily linearly representable over every field. One obstruction is the size of the field. For example, the uniform matroid $U_{2,n}$ is linearly representable over a field $\F$ if and only if $|\F|\geq n-1$.

Clearly, if a matroid is linearly representable over a field $\F$, then it is linearly representable over every extension $\E/\F$. So it is natural to ask: For a given linear matroid $M$ and a field $\F$, does there always exist an extension $\E/\F$ over which $M$ is linearly representable?

The answer to this question is negative, because many linear matroids are not linearly representable over every field of certain characteristics. For example, the non-Fano matroid is linearly representable over $\F$ if and only if $\Char(\F)\neq 2$, and the Fano matroid is linearly representable over $\F$ if and only if $\Char(\F)=2$. This motivates the definition of characteristic sets. The characteristic set of a matroid $M$ is the set $$\chi(M):=\set{k| M {\rm\ is\ linearly\ representable\ over\ some\ field\ of\ characteristic\ } k}.$$

Lazarason~\cite{Laz} constructed a family of matroids that contains, for every prime $p$, a matroid with characteristic set $\set{p}$, i.e., linearly representable only over fields of characteristic $p$. Reid constructed a family of rank-$3$ matroids with the same property. These families of matroids are now called the Lazarason matroids and the Reid geometries respectively.

Another interesting and important example is a family of matroids introduced by Dowling~\cite{Dowlinga,Dowlingb}, based on finite groups. This family of matroids, now called Dowling geometries, contains a matroid of every rank $\geq 3$ for every finite group.\footnote{Dowling matroids can also be defined for rank $<3$, but in this case the Dowling matroids are nothing but the uniform matroids of the same rank. So in this paper we refer only to Dowling geometries of rank $\geq 3$.} Dowling showed that two Dowling geometries are isomorphic if and only if the underlying groups are isomorphic. He also showed that a Dowling geometry of a group $G$ is linearly representable over a field $\F$ if and only if $G$ is a subgroup of $\F^{\times}$ (the group of invertible elements of $\F$). So in terms of characteristics sets, the characteristic set of a Dowling geometry of a non-cyclic group is the empty set, and the characteristic set of a Dowling geometry of a cyclic group $C_m$ is the set of primes not dividing $m$.

\paragraph{Algebraic matroids.}
The class of algebraic matroids comes from the algebraic dependence in transcendental field extensions. In algebraic representations, to each element in the ground set one associates an element of a field $\F$ such that a set of elements is dependent if and only if the corresponding elements are algebraically dependent over a field $\K$, where $\K\subseteq \F$ is a fixed sub-field. If such a representation exists, the matroid is called algebraically representable over $\K$. 

If a matroid is linearly representable over a field, then it is also algebraically representable over the same field (cf. \cite[Proposition 6.7.10]{Oxl11}). Therefore, the class of algebraic matroids extends the class of linear matroids. However, not every algebraic matroid is linear, e.g., the non-Pappus matroid is algebraic~\cite{Lind83}, but not linear. 

It is known that not all matroids are algebraic. For example, Ingleton and Main~\cite{IngMain} showed that the V\' amos matroid is not algebraically representable over any field. It is also known that the class of algebraic matroids is minor closed. However, many fundamental questions regarding algebraic matroids remain open. For example, it is still unknown if the dual of an algebraic matroid is also algebraic. There is also no known general algorithm for determining if a given matroid is algebraic, or even algebraic over a given field.

\paragraph{Algebraic characteristic sets.}
Algebraic characteristic sets are defined similarly to linear characteristic sets. An alternative equivalent definition is the set $$\chi_A(M):=\set{k| M {\rm\ is\ algebraically\ representable\ over\ the\ prime\ field\ of\ characteristic\ } k}.$$\remove{of primes $p$ such that the matroid is algebraically representable over $\F_p$, with $0$ in the characteristic set if the matroid is algebraically representable over $\Q$.} The equivalence of the definitions follows from the following fundamental result: if a matroid $M$ is algebraically representable over a field $\F$, then it is also algebraically representable over the prime field of $\F$. It was conjectured by Piff~\cite{Piff} and proved by Lindst\" om~\cite{Lind89}.

\remove{If a matroid is linearly representable over a field $\F$, then it is also algebraically representable over $\F$ (cf. \cite[Chapter 6]{Oxl11}). Therefore, }The linear characteristic set of a matroid is always contained in its algebraic characteristic set. Using derivations, Ingleton~\cite{Ing71b} showed that if a matroid is algebraically representable over a field $\K$ of characteristic $0$, then it is linearly representable over some extension field $\E$ of $\K$ (in fact, as Ingleton pointed out, the proof was already known, e.g., \cite{Zar58}, but not stated in the language of matroids). So $0$ is in the algebraic characteristic set if and only if it is in the linear characteristic set.

But for positive characteristics the situation changes. For example, Ingleton~\cite{Ing71b} presented an algebraic matroid that contains both the Fano and the non-Fano matroids as minors. Lindstr\" om~\cite{Lind83} showed that the algebraic characteristic set of the non-Pappus matroid contains all primes.

After the work of Ingleton~\cite{Ing71b}, derivations have been one of the main tools in the study of algebraic representability and algebraic characteristic sets. Using derivations, Lindstr\" om~\cite{Lind85} showed that for every prime $p$, the corresponding Lazarason matroid has algebraic characteristic set $\set{p}$. Gordon~\cite{Gordon88} showed a similar result for the Reid geometries. He also constructed matroids with finite algebraic characteristic sets of cardinality greater than one.

\paragraph{Multi-linear matroids.}
The class of multi-linear matroids developed more from practical reasons. Multi-linear matroids are used in cryptography, in coding theory, and in network coding, e.g., \cite{SA98,Aner,Rou10}.

Multi-linear representations were explicitly defined by Simonis and Ashikmin~\cite{SA98}. Multi-linear representations are defined similarly to linear representations, except that each element is mapped to $k$ vectors, for some fixed $k\in \N$. The matroid rank of a set of elements is obtained by dividing the rank of the relevant subspace by $k$. It is important to note that such a rank function is the rank function of matroid if and only if it is integer valued, which might not be the case when $k>1$. If the rank function is integer valued, the resulting matroid is called multi-linearly representable.

There are multi-linearly representable matroids that are not linear. Simonis and Ashikmin~\cite{SA98} showed that the non-Papus matroid is multi-linearly representable over $\F_3$. Pendavingh and van Zwam~\cite{Zwam13} constructed another multi-linear representation of the non-Papus matroid, over $\Q$. They also showed that the rank-3 Dowling geometry associated with the quaternion group is multi-linearly representable. By combining with the ternary Reid geometry, they constructed a multi-linear matroid that is not representable even over division rings. Beimel et al.~\cite{Aner} showed that the rank-3 Dowling geometry of a group $G$ is multi-linearly representable if and only if $G$ is fixed-point free (see Section \ref{subsec:rep} for the precise definition of this term). By applying this to properly chosen groups, they showed that for every prime $p$ there exists a matroid that admits a $p$-linear representation but does not admit a $k$-linear representation for every $k<p$.

\paragraph{Generalizations of multi-linear representations.}
Brickel and Davenport~\cite{BD91} showed an important natural correlation between a cryptographic primitive called ideal secret-sharing schemes and a class of matroids, now called secret-sharing matroids. This class of matroids is also studied in coding theory, under the name almost affinely representable matroids~\cite{SA98}. Mat\' u\v s~\cite{Mat99} referred to these matroids as partition representable matroids, because they can be defined using partitions of a set satisfying certain constraints. 

The class of secret-sharing matroids contains the class of multi-linear matroids~\cite{SA98}, but it is not known if the class of multi-linear matroids is a proper sub-class. 

Not all matroids are secret-sharing. Specifically, Seymour~\cite{Sey} showed that the V\' amos matroid is not secret-sharing.  Mat\' u\v s~\cite{Mat99} showed that any matroid containing both the Fano and the non-Fano matroids as minors is not secret-sharing. Thus, he showed that there exists an algebraic matroid that is not secret sharing, and therefore also not multi-linear -- the algebraic non-linear matroid given by Ingleton in ~\cite{Ing71b}. Mat\' u\v s~\cite{Mat99} then asked the following natural question: is every secret-sharing matroid algebraic?

Another intriguing form of matroid representations, defined recently by Pendavingh and van Zwam~\cite{Zwam13}, is skew partial field representations. Pendavingh and van Zwam showed that every multi-linearly representation is a special case of a skew partial field representation. So all multi-linear matroids are also skew partial field representable. It is not known if there exists a matroid representable over a skew partial field that is not multi-linear. The relation between skew partial field representable matroids and secret-sharing matroids is also unknown. In their paper, Pendavingh and van Zwam formulated the following open question: is the class of skew partial field representable matroids contained in the class of algebraic matroids?

\paragraph{Our results.}
We will show that there exists a multi-linearly representable, non-algebraic matroid. Since multi-linear matroids are secret-sharing matroids and also skew-partial field representable, this result answers negatively both the question of Mat\' u\v s~\cite{Mat99} and that of Pendavingh and van Zwam~\cite{Zwam13}.

The matroid is constructed as a direct sum of the Reid Geometry $R_p$, with $p=7$, and a new matroid that we introduce, $Q^{NF}_3(\mathsf{Q}_8)$, which contains both the non-Fano matroid and the rank-3 Dowling geometry of the quaternion group as minors.

\paragraph{Organization.} In Section \ref{sec:prelim} we give the definitions and review some necessary basic background material (the reader who feels comfortable with the definitions may skip the appropriate subsections): In Section \ref{subsec:def} we define multi-linear representations (Section \ref{subsec:defml}) and algebraic representations (Section \ref{subsec:defalg}). In Section \ref{subsec:rep} we define fixed-point free representations and groups. In Section \ref{app:dowling} we give the definition of the rank-3 Dowling group geometries\remove{ as it is given in \cite{Aner}}. In Section \ref{sec:derivations} we discuss derivations and their use in the study of algebraic matroids -- we recall the definition and some basic facts (Section \ref{subsec:der}), and then we describe how derivations in certain algebraic representations induce linear representations of (possibly different) matroids (Section \ref{subsec:induced}). 

In Section \ref{sec:known} we recall some of the known results regarding algebraic and multi-linear representability of the Reid geometries and Dowling geometries. 

In Section \ref{sec:main} we state and prove our main theorem. The proofs of two technical lemmas are left to Section \ref{sec:proofs}.

\paragraph{Notation.} When referring to $k$-linear representations, we will frequently use block matrices. To distinguish these block matrices, they will be inside square brackets, or in bold letters (e.g. $\mathbf{A}= \left[\begin{matrix}
A&B\\C&D\end{matrix}\right]$, where $A,B,C,D$ are matrices). In all the proofs and examples all the blocks are square matrices\remove{of size $k\times k$}. For a matrix $A$, we denote  by $A_i$ the $i^{th}$ column of $A$. 
We will treat isomorphic matroids as equal (i.e., if $M\cong N$ we will often write $M=N$ instead). 

\section{Preliminaries}\label{sec:prelim}
\subsection{Extending Linear Representations}\label{subsec:def}
We assume the reader is familiar with matroid basics, such as rank function, linear representations, direct sums, minors, simple matroids, geometric representations of matroids, and with the Fano and the non-Fano matroids. A good introduction to Matroid Theory can be found in \cite{Oxl11}. We therefore proceed directly to defining multi-linear representations and algebraic representations.

\subsubsection{Multi-Linear Representations}\label{subsec:defml}

Multi-linear representations are defined similarly to linear representations, except that each element is associated with a set of $k$ vectors, for some fixed $k$. We will use the definition with matrices.\remove{, instead of the somewhat more common definition, with grassmannians.}

\begin{definition}\label{def:$k$-linear}
Let $M=(E=\{1,...,N\},r)$ be a matroid, $\F$ a field and $k\in \N$. A {\em $k$-linear representation} of $M$ over $\F$ is a matrix $A$ with $k\cdot N$ columns $A_1,...,A_{k\cdot N}$ ($N\in \N$) such that the rank of every set $X=\set{i_{1},...,i_{j}}\subset E$ satisfies $$r(X)=\frac{\dim(\Span(U_{i_{1}}\cup ...\cup U_{i_{j}}))}{k},$$ where $$U_{\ell}=\{A_{(\ell-1)\cdot k+1},A_{(\ell-1)\cdot k+2},...,A_{\ell\cdot k}\}.$$\remove{ for $1\leq \ell \leq N$} If such a representation of $M$ exists then $M$ is called {\em $k$-linearly representable} and $A$ is its $k$-linear representation over $\F$. We will say that a matroid is {\em multi-linear} if it is $k$-linearly representable for some $k\in \N$ over some field $\F$.
\end{definition} 
When $k=1$ this is the regular definition of linear representations. We note that although a matrix with $N$ columns always defines a linear matroid with $N$ elements, for $k>1$ a matrix with $Nk$ columns does not necessarily define a $k$-linear matroid, because when we divide by $k$ the value is not necessarily an integer. However, if the rank of every relevant sub-matrix is a multiple of $k$, then the matroid axioms are satisfied, and we have a $k$-linear matroid.

\subsubsection{Algebraic Representations}\label{subsec:defalg} 
Algebraic matroids come from transcendental field extensions:
\begin{definition}\label{def:algebraic}
Let $M=(E=\{1,...,N\},r)$ be a matroid, $\K$ a field and $\K\subset \F$ a field extension. An {\em algebraic representation} of $M$ over $\K$ is a mapping $\pi\colon E\to \F$ such that the rank of every set $X=\set{i_{1},...,i_{j}}\subset E$ satisfies $r(X)=\deg_{tr}(\K(\pi(i_{1}),...,\pi(i_{j}))/\K)$. If such a representation of $M$ exists then $M$ is called {\em algebraically representable} (or simply {\em algebraic}) and $\pi$ is its algebraic representation over $\K$.
\end{definition}
It is well known that a rank function defined this way (transcendental degree) always satisfies the matroid axioms, and that every linear matroid is algebraic. The proofs can be found, for example, in \cite[Chapter 6.7]{Oxl11}.

A useful fact in studying algebraic representations, which we use in our main theorem, is that we can always assume that $\K$ is the prime field (i.e., $\Q$ or $\F_p$ for some prime $p$). This was conjectured by Piff~\cite{Piff}, who reduced the problem to showing that if $M$ is algebraically representable over $\F(t)$, with $t$ transcendental, then $M$ is algebraically representable over $\F$. The proof was later completed by Lindstr\" om \cite{Lind89}.
\begin{ttheorem}[\textbf{Piff's conjecture}]\label{th:basefield}
If a matroid is algebraic over $\K$, then it is also algebraic over the prime subfield of $\K$.
\end{ttheorem}

So from now on, we only consider algebraic representations over prime fields.\remove{may assume that if $\pi\colon E\to \F$ is an algebraic representation then the transcendence degree is over the prime field of $\F$.}

\subsection{Fixed-Point Free Representations}\label{subsec:rep}
In this section we quickly recall some basic definitions in representations theory, including the somewhat less common definition of fixed-point free representations.

\begin{definition}\label{def:representations}
Let $G$ be a finite group and $\F$ a field. 
A {\em representation} of $G$ is a group homomorphism $\rho\colon G\rightarrow \GL_k(\F)$. \remove{ (the group of $k\times k$ invertible matrices). }
The {\em dimension} or {\em degree} of the representation is $k$. 
A representation is called {\em faithful} if it is injective. 
A representation $\rho\colon G\rightarrow \GL_k(\F)$ is {\em fixed-point free} if for every $e\neq g\in G$ the field element 1 is not an eigenvalue of $\rho (g)$, i.e., $\rho(g)\cdot v\neq v$ for every $g\neq e$,  $v\neq 0$. 
A {\em fixed-point free group} is one which admits a fixed-point free representation.
\end{definition}
Fixed-point free representations are obviously faithful, but the converse is false in general. Note that not every representation of a fixed-point free group $G$ is fixed-point free, even if the representation is faithful. For example, cyclic groups are fixed-point free but also have faithful non-fixed-point free representations:
\begin{example}\label{ex:cyclicrep}
Let $G=C_m$ be the cyclic group with $m$ elements. Denote $\zeta=e^{\frac{2\pi i}{m}}$. If $\rho\colon G\to \GL_2(\C)$ is defined by $\rho(k)=\left(\begin{matrix} \zeta^k&0\\0&1\end{matrix}\right)$ then $\rho$ is faithful\remove{(because $i\neq k\Rightarrow \rho(i)\neq \rho(k)$)} but not fixed-point free because $\left(\begin{matrix}
\zeta^k&0\\0&1\end{matrix}\right)\left(\begin{matrix}
0\\1\end{matrix}\right)=\left(\begin{matrix}
0\\1\end{matrix}\right)$ \remove{(and this should only happen for $k=0$)}. However, if we define $\rho(k)=\left(\begin{matrix} \zeta^k&0\\0&\zeta^k\end{matrix}\right)$ then $\rho$ is fixed-point free.\remove{, because if $k\neq 0$ then 1 is not an eigenvalue of $\left(\begin{matrix} \zeta^k&0\\0&\zeta^k\end{matrix}\right)$.} The group $\C_m$ also has a fixed-point free representation of dimension 1, given by $\rho(k)=(\zeta^k)$.
\end{example}
\begin{observation}\label{ob:subtraction}
If $\rho\colon G\to \GL_k(\F)$ is a fixed-point free representation, and $g,h\in G$ are distinct, then $\rho(g)-\rho(h)$ is invertible.
\end{observation}
\begin{proof}
If $\rho(g)-\rho(h)$ is not invertible, then there exists a vector $\vec{v}\neq 0$ such that $(\rho(g)-\rho(h))\vec{v}=0$, implying $\rho(g)\vec{v}=\rho(h)\vec{v}$. So $\vec{v}=\rho(g)^{-1}\rho(h)\vec{v}=\rho(g^{-1}h)\vec{v}$. Thus, $\vec{v}$ is an eigenvector with eigenvalue $1$ of $\rho(g^{-1}h)$, a contradiction. 
\end{proof}

\subsection{Rank-3 Dowling Group Geometries}\label{app:dowling}
The Dowling group geometries were introduced by Dowling~\cite{Dowlinga,Dowlingb}. The rank-$r$ Dowling geometry of a group $G$ is denoted by $Q_r(G)$. The general construction of $Q_r(G)$ can be found in \cite{Dowlingb} or \cite[Chapter 6.10]{Oxl11}. We use a slightly different construction of the rank-3 Dowling group geometries, following \cite{Aner}. The resulting matroid in the usual construction is isomorphic to this construction by a relabelling of the elements.

Let $G=\{e=g_{1},g_{2},...,g_{n}\}$ be a finite group. The rank-3 Dowling geometry of $G$, denoted $Q_{3}(G)$, is a matroid of rank 3 on the set $E=\{p_1,p_2,p_3,g^{(1)}_1,...,g^{(1)}_n,g^{(2)}_1,..,g^{(2)}_n,g^{(3)}_1,...,g^{(3)}_n\}$. So there are 3 ground set elements $p_1,p_2,p_3$, called the joints, which are not related to the group, and for every element $g\in G$, there are 3 elements in the ground set of the matroid $g^{(1)},g^{(2)},g^{(3)}\in E$.

\begin{figure}[h]
$$
\setlength{\unitlength}{3cm}
\begin{picture}(1.5,1)
\put(0,0){\line(1,1){1}}
\put(2,0){\line(-1,1){1}}
\put(0,0){\line(1,0){2}}

\put(0,0){\circle*{.05}}
\put(-0.15,-0.01){$p_1$}
\put(1,1){\circle*{.05}}
\put(1.0,1.07){$p_2$}
\put(2,0){\circle*{.05}}
\put(2.05,0){$p_3$}

\put(0.1,0.1){\circle*{.05}}
\put(-0.05,0.1){$g^{(1)}_1$}
\put(0.2,0.2){\circle*{.05}}
\put(0.25,0.15){$g^{(1)}_2$}
\put(0.3,0.3){\circle*{.05}}
\put(0.15,0.3){$g^{(1)}_3$}

\put(0.45,0.45){\circle*{.03}}
\put(0.6,0.6){\circle*{.03}}
\put(0.75,0.75){\circle*{.03}}

\put(0.9,0.9){\circle*{.05}}
\put(0.75,0.92){$g^{(1)}_n$}

\put(1.1,0.9){\circle*{.05}}
\put(1.14,0.9){$g^{(2)}_1$}
\put(1.2,0.8){\circle*{.05}}
\put(1.24,0.8){$g^{(2)}_2$}
\put(1.3,0.7){\circle*{.05}}
\put(1.34,0.7){$g^{(2)}_3$}

\put(1.45,0.55){\circle*{.03}}
\put(1.6,0.4){\circle*{.03}}
\put(1.75,0.25){\circle*{.03}}

\put(1.9,0.1){\circle*{.05}}
\put(1.94,0.1){$g^{(2)}_n$}

\put(1.8,0){\circle*{.05}}
\put(1.7,-0.15){$g^{(3)}_1$}
\put(1.6,0){\circle*{.05}}
\put(1.5,-0.15){$g^{(3)}_2$}
\put(1.4,0){\circle*{.05}}
\put(1.3,-0.15){$g^{(3)}_3$}

\put(1.1,0){\circle*{.03}}
\put(0.8,0){\circle*{.03}}
\put(0.5,0){\circle*{.03}}

\put(0.2,0){\circle*{.05}}
\put(0.2,-0.15){$g^{(3)}_n$}

\end{picture}$$
\caption{A geometric representation of the Rank-3 Dowling geometry with the lines corresponding to the sets $\set{g^{(1)}_i,g^{(2)}_j,g^{(3)}_{\ell}}$ such that $g_j\cdot g_i\cdot g_{\ell}=e$ missing.\label{fig:Dowling}}
\end{figure}
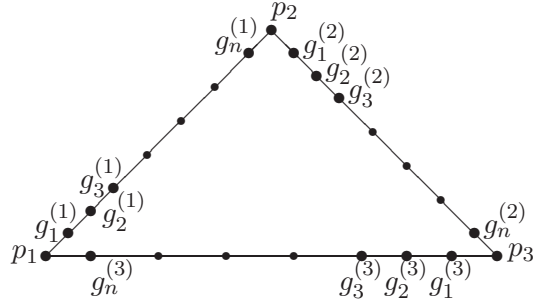

Figure \ref{fig:Dowling} is a geometric representation of $Q_{3}(G)$, with additional lines that go through the points $g^{(1)}_i$, $g^{(2)}_j$, $g^{(3)}_{\ell}$  if and only if  $g_j\cdot g_i\cdot g_{\ell}=e$ (e.g., there is always a line that goes through $g^{(1)}_1,g^{(2)}_1,g^{(3)}_1$ since $g_1=e$ and $e\cdot e\cdot e=e$).

Notice that in the geometric representation of $Q_3(G)$, every point of $Q_3(G)$ lies on one of the following lines: $\set{p_1,p_2}$,$\set{p_2,p_3}$, $\set{p_1,p_3}$. We refer to these lines as the {\em edges} of $Q_3(G)$.

So the set of bases of $Q_{3}(G)$ is every set of 3 elements except
\begin{enumerate}
\item A set of 3 points on the same edge.
\item The sets $\{g^{(1)}_i,g^{(2)}_j,g^{(3)}_{\ell}\}$ when $g_{j}\cdot g_{i}\cdot g_{\ell} = e$.
\end{enumerate}

Figures \ref{fig:DTM}$(a)$ and \ref{fig:DTM}$(b)$ represent the Dowling Geometries of the trivial group and of the cyclic group $C_2$ respectively.

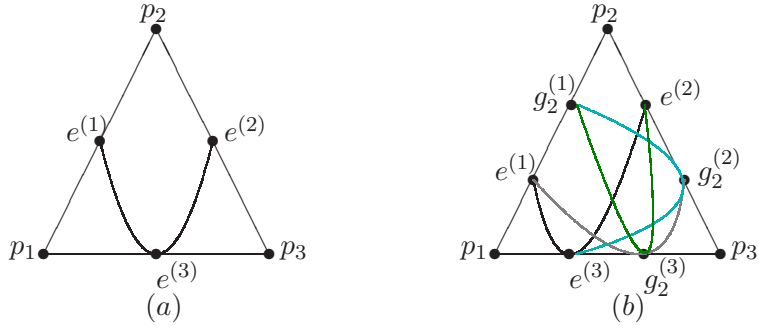
\begin{figure}[h]
$$
\setlength{\unitlength}{3cm}
\begin{picture}(2,1.1) 
\put(0,0.1){\line(1,2){.5}}
\put(1,0.1){\line(-1,2){.5}}
\put(0,0.1){\line(1,0){1}}
\put(0,0.1){\circle*{.05}}
\put(-0.15,0.1){$p_1$}
\put(.5,1.1){\circle*{.05}}
\put(.43,1.15){$p_2$}
\put(1,0.1){\circle*{.05}}
\put(1.05,0.1){$p_3$}
\put(.25,.6){\circle*{.05}}
\put(.1,.6){$e^{(1)}$}
\put(.5,0.1){\circle*{.05}}
\put(.5,-0.05){$e^{(3)}$}
\put(.75,0.6){\circle*{.05}}
\put(.8,0.6){$e^{(2)}$}
\qbezier(0.25,0.6)(.5,-.4)(.75,.6)
\put(.45,-.17){$(a)$}
\end{picture}
\begin{picture}(1,1.1)
\put(0,0.1){\line(1,2){.5}}
\put(1,0.1){\line(-1,2){.5}}
\put(0,0.1){\line(1,0){1}}
\put(0,0.1){\circle*{.05}}
\put(-0.15,0.1){$p_1$}
\put(.5,1.1){\circle*{.05}}
\put(.43,1.15){$p_2$}
\put(1,0.1){\circle*{.05}}
\put(1.05,0.1){$p_3$}
\put(.17,.43){\circle*{.05}}
\put(.0,.43){$e^{(1)}$}
\put(.34,.76){\circle*{.05}}
\put(.17,.76){$g^{(1)}_2$}
\put(.33,0.1){\circle*{.05}}
\put(.33,-0.05){$e^{(3)}$}
\put(.66,0.1){\circle*{.05}}
\put(.66,-0.05){$g^{(3)}_2$}
\put(.67,0.76){\circle*{.05}}
\put(.73,0.76){$e^{(2)}$}
\put(.84,0.43){\circle*{.05}}
\put(.90,0.43){$g^{(2)}_2$}
\qbezier(.17,.43)(.33,-.37)(.67,0.76)
{\color{mygray}\qbezier(.17,.43)(0.77,-0.23)(.84,0.43)}
{\color{myblue}\qbezier(.34,.76)(1.27,0.37)(.33,0.1)}
{\color{myred}\qbezier(.3,.76)(0.75,-0.55)(.6,0.76)}
\put(.45,-.17){$(b)$}
\end{picture}$$
\caption{Geometric Representation of the matroids $Q_3(\{e\})$ and $Q_3(C_2)$.\label{fig:DTM}}
\end{figure}

\subsection{Derivations and their Induced Matroid}\label{sec:derivations}
In this section we first briefly review some properties of derivations. We then describe how derivations are used to study algebraic matroids. Specifically, we describe the linear matroid induced by the derivations. In this paper, we will need only $\K$-linear derivations of a field extension $\F/\K$.

\subsubsection{Derivations}\label{subsec:der}
Derivations are linear maps that satisfy the Leibniz rule.
\begin{definition}
Let $\K \subset \F$ be a field extension. A {\em derivation} of $\F$ over $\K$ is a mapping $D\colon\F\to \F$ such that 
\begin{enumerate}
\item $\forall y,z\in \F,D(y+z)=D(y)+D(z)$,
\item $\forall y,z\in \F,D(yz)=zD(y)+yD(z)$,
\item $\forall a\in \K, D(a)=0$.
\end{enumerate}
\end{definition}

We will need the following Lemma and basic properties of derivations, which can be found, for example, in\cite[Chapter VIII, Section 5]{Lang}.
\begin{lemma}\label{cor:der2}
If $\F\subset \F(y)$ is a field extension and $\bar{D}\colon\F(y)\to \F(y)$ is an extension of $D$ to $\F(y)$ (i.e., $\forall a\in\F,\bar{D}(a)=D(a)$), then for $f\in \F[Y]$ that vanishes on $y$ we have $\bar{D}(y)\cdot f'(y)+f^D(y)=0$, where  $f'$ is the formal derivative of $f$.
\end{lemma}

The set of all derivations of $\F$ over $\K$, with the natural definition of addition and scalar multiplication forms a vector space denoted $\Der_{\K}(\F)$. If $\F=\K(x_1,...,x_n)$, such that $\set{x_1,...,x_n}$ are algebraically independent over $\K$, then the derivations $\set{D_i:=\frac{\partial}{\partial x_i}}_{i=1}^n$, i.e., $D_i(x_j)=\delta_{i,j}$ (the Kronecker delta), are a basis of $\Der_{\K}(\F)$.

If $\F\subseteq \E$ is a finite separable extension, then each $D\in\Der_{\K}(\F) $ extends uniquely to a derivation $\bar{D}\colon\E\to\E$, and the derivations $\set{\bar{D_i}}_{i=1}^n$ form a basis of $\Der_{\K}(\E)$.

\subsubsection{The Matroid Induced by the Derivations}\label{subsec:induced}
Let $M=(E=\set{1,\dots,N},r)$ be a matroid with a basis $\set{1,\dots,n}$. Suppose $\pi\colon E\to \E=\F_p(\pi(E))$ is an algebraic representation of $M$ over $\F_p$. Then $\set{\pi(1),\dots,\pi(n)}$ are algebraically independent over $\F_p$, and $\F_p(\pi(1),\dots,\pi(n))\subseteq\E$ is an algebraic extension. If we suppose further that this extension is separable, then the derivations $\bar{D}_i$ are uniquely defined for $1\leq i\leq n$, so this induces a linear representation of a (possibly different) matroid $\bar{M}=(E,\bar{r})$, with the same ground set, in the following way: Denote the gradient of an element $y\in \E$ by $\nabla(y):=(\bar{D_1}(y),...,\bar{D_n}(y))\in \E^n$. Then the matrix with the $i^{th}$ column being $\nabla(\pi(i))$ is a linear representation over $\E$ of $\bar{M}$.

Every dependent set in $M$ is also dependent in $\bar{M}$ (see for example~\cite{Ing71b}). As Ingleton~\cite{Ing71b} showed, in characteristic 0 the converse is also true, so $M=\bar{M}$. Therefore, in characteristic 0 every algebraic matroid is linearly representable. However, if the characteristic is positive then the equality $M=\bar{M}$ may fail, as the following example shows:
\begin{example}\label{ex:alg2lin}
Let $\K=\F_3$ and $\E=\F=\K(x_1,x_2,x_3)$ with $x_1,x_2,x_3$ independent variables. On the set $E=\set{1,\dots,6}$ define the mapping $\pi\colon E\to \E$ by:
\begin{center}
\begin{tabular}{c||c|c|c|c|c|c}
$e$&1&2&3&4&5&6\\ \hline \hline
$\pi(e)$ & $x_1$ & $(x_2)^3$ & $x_3$ & $x_1+(x_2)^3$ & $x_1+x_3$ & $x_1+(x_2)^3+2x_3$\\ 
\end{tabular}\,\,\,\,\,{\raisebox{-2ex}{.}}
\end{center}

For $X\subseteq E$, set $r(X):=\deg_{tr}(\K(\pi(X))/\K)$. This defines an algebraic matroid $M$ whose geometric representation is Figure \ref{fig:der} ($a$). The corresponding gradients are 
\begin{center}
\begin{tabular}{c||c|c|c|c|c|c}
$e$&1&2&3&4&5&6\\ \hline \hline
$\nabla(\pi(e))$ & $(1,0,0)$ & $(0,0,0)$ & $(0,0,1)$ & $(1,0,0)$ & $(1,0,1)$ & $(1,0,2)$\\ 
\end{tabular}\,\,\,\,\,{\raisebox{-2ex}{,}}
\end{center}
which induce the linear matroid $\bar{M}$. The corresponding geometric representation is Figure \ref{fig:der} ($b$). As one can see, there are dependent sets in $\bar{M}$ which were independent in $M$, e.g. $\set{2},\set{1,4},\set{1,3,6}$.
\end{example}
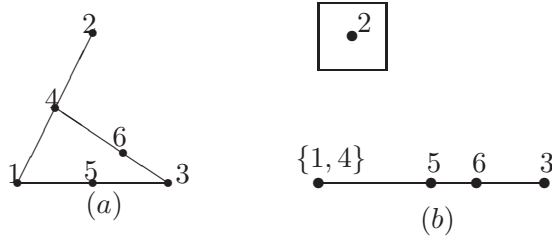
\begin{figure}
$$
\setlength{\unitlength}{2cm}
\begin{picture}(2,1) 
\put(0,0){\line(1,2){.5}}
\put(0,0){\line(1,0){1}}
\put(1,0){\line(-3,2){.75}}
\put(0,0){\circle*{.05}}
\put(-0.07,0){$1$}
\put(.5,1){\circle*{.05}}
\put(.43,1){$2$}
\put(1,0){\circle*{.05}}
\put(1.05,0){$3$}
\put(.25,.5){\circle*{.05}}
\put(.18,.5){$4$}
\put(.5,0){\circle*{.05}}
\put(.44,0.02){$5$}
\put(.7,.2){\circle*{.05}}
\put(.63,.25){$6$}
\put(.45,-.2){$(a)$}
\end{picture}
\setlength{\unitlength}{3cm}
\begin{picture}(1,1) 
\put(0,0){\line(1,0){1}}
\put(0,0){\circle*{.05}}
\put(-0.1,0.07){$\{1,4\}$}
\put(1,0){\circle*{.05}}
\put(0.98,0.05){$3$}
\put(.5,0){\circle*{.05}}
\put(.48,0.05){$5$}
\put(.7,0){\circle*{.05}}
\put(.68,0.05){$6$}
\put(.45,-.2){$(b)$}

\put(0,0.5){\line(1,0){0.3}}
\put(0,0.5){\line(0,1){0.3}}
\put(0.3,0.8){\line(-1,0){0.3}}
\put(0.3,0.8){\line(0,-1){0.3}}
\put(.15,0.65){\circle*{.05}}
\put(.17,0.67){$2$}

\end{picture}
$$
\caption{\label{fig:der}An algebraic matroid and its induced derivations matroid.}
\end{figure}

Given an algebraic representation $\pi$ of $M$ over $\F_p$, one can construct another algebraic representation $\pi'$ over $\F_p$ in the following way:
\begin{lemma}\label{lem:basicreplace}
Let $M$ be an algebraic matroid on ground set $E=\set{1,\dots,n}$ of rank $r$, $\pi\colon E\to\E$ an algebraic representation of $M$ over $\F_p\subseteq\E$. For a function $m\colon E\to\Z$, we set $\pi'(i):=(\pi(i))^{p^{m(i)}}\in \bar{\E}$. Then $\pi' \colon E\to\F_p(\pi'(E))$ is an algebraic representation of $M$.
\end{lemma}
The proof is simple and therefore omitted. If $m$ is not the zero function, then clearly $\pi'\neq\pi$. It is important to notice that the matroid induced by the derivations depends on the specific representation.
\begin{example}\label{ex:alg2lin2}
Let $\K=\F_3$, $\E_1=\E_2=\F=K(x_1,x_2)$, and $\E_3=\K(x_1,x_2,(x_1+x_2)^{\frac{1}{3}})$. Notice that $\F\subset \E_3$ is not a separable extension, since the minimal polynomial of $(x_1+x_2)^{\frac{1}{3}}$ in $\F[X]$ is $f=X^3-(x_1+x_2)$, so $f'=0$. On the set $E=\set{1,2,3}$ define 3 algebraic representations, $\pi_i\colon E\to \E_i,1\leq i\leq 3$, as follows:
\begin{center}
\begin{tabular}{c||c|c|c}
$e$&1&2&3\\ \hline \hline
$\pi_1(e)$ & $x_1$ & $x_2$ & $x_1+x_2$\\ \hline
$\pi_2(e)$ & $x_1$ & $x_2$ & $(x_1+x_2)^3$\\ \hline
$\pi_3(e)$ & $x_1$ & $x_2$ & $(x_1+x_2)^{\frac{1}{3}}$\\ 
\end{tabular}\,\,\,\,\,{\raisebox{-5ex}{.}}
\end{center}
From Lemma \ref{lem:basicreplace}, the mappings $\pi_1$, $\pi_2$, and $\pi_3$ are all algebraic representations of the same matroid, namely $U_{2,3}$, whose geometric representation is Figure \ref{fig:der2} ($a$). By calculation, as in Example \ref{ex:alg2lin}, we can see that the geometric representation of the derivations matroid for $\pi_1$ is Figure \ref{fig:der2} ($a$) ,i.e., it is equal to the original matroid. For $\pi_2$, the geometric representation is Figure \ref{fig:der2} ($b$). 

Notice that in the derivations matroid of $\pi_2$, the singleton $\set{3}$ is dependent, because the gradient of $\pi_2(3)$ is the zero vector.

For $\pi_3$, the derivations matroid makes no sense because $\E_3/\F$ is not a separable extension. Indeed, if $D_1$ was extendable to $\bar{D}_1\in \Der_\K(\E_3)$, then $1=\bar{D}_1(x_1+x_2)=\bar{D}_1(((x_1+x_2)^{\frac{1}{3}})^3)=3((x_1+x_2)^{\frac{1}{3}})^2\bar{D}_1((x_1+x_2)^{\frac{1}{3}})=0$ would be a contradiction.

In the proof of our main theorem we will use Lemma \ref{lem:basicreplace} to avoid representations such as $\pi_3$ (inseparable extension) and $\pi_2$ (with dependent singletons in the induced matroid).
\end{example}

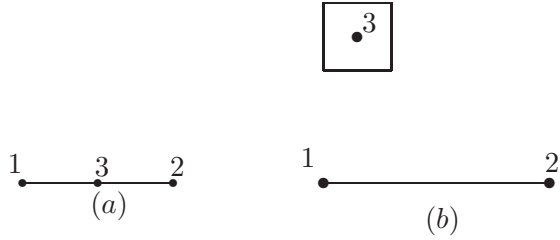
\begin{figure}
$$
\setlength{\unitlength}{2cm}
\begin{picture}(2,1) 
\put(0,0){\line(1,0){1}}
\put(0,0){\circle*{.05}}
\put(-0.1,0.07){$1$}
\put(1,0){\circle*{.05}}
\put(0.98,0.05){$2$}
\put(.5,0){\circle*{.05}}
\put(.48,0.05){$3$}
\put(.45,-.2){$(a)$}
\end{picture}
\setlength{\unitlength}{3cm}
\begin{picture}(1,1) 
\put(0,0){\line(1,0){1}}
\put(0,0){\circle*{.05}}
\put(-0.1,0.07){$1$}
\put(1,0){\circle*{.05}}
\put(0.98,0.05){$2$}
\put(.45,-.2){$(b)$}

\put(0,0.5){\line(1,0){0.3}}
\put(0,0.5){\line(0,1){0.3}}
\put(0.3,0.8){\line(-1,0){0.3}}
\put(0.3,0.8){\line(0,-1){0.3}}
\put(.15,0.65){\circle*{.05}}
\put(.17,0.67){$3$}
\end{picture}
$$
\caption{\label{fig:der2}The algebraic matroid $U_{2,3}$ and one of its induced derivations matroids.}
\end{figure}

We note that for the matroids in Examples \ref{ex:alg2lin} and \ref{ex:alg2lin2}, one can clearly choose algebraic representations in which the matroids induced by the derivations coincide with the original matroid. However, this is not always the case for an algebraic representation. Specifically, for a derivations matroid $\bar{M}$ of non-linear algebraic matroid $M$, we necessarily have $\bar{M}\neq M$. This will also be the case in the
proof of our main theorem, where the original matroid will contain the non-Fano matroid as a minor, while the derivations matroid will contain the Fano matroid as a minor.

\section{Some Known Results}\label{sec:known}
In this section we summarize some known results about \remove{the representability of }2 families of matroids: the Reid geometries and the Dowling geometries. In Section \ref{sec:main} we construct a non-algebraic matroid containing the Reid geometry $R_7$ and the rank-3 Dowling geometry of the quaternion group\remove{, $Q_3(\mathsf{Q}_8)$,} as minors.

\paragraph{The Reid Geometries.}
For a prime $p$, let $R_p$ be the matroid obtained from the following $k$-linear representation over a field $\F$, with $\Char(\F)=p$:
$$\mathbf{B}=\left[\begin{matrix}
I_k&0&0&I_k&I_k&0&0&\dots&0\\
0&I_k&0&I_k&0&I_k&I_k&\dots&I_k\\
0&0&I_k&0&I_k&I_k&2\cdot I_k&\dots &(p-1)\cdot I_k
\end{matrix}
\right.
\left.\begin{matrix}
&I_k&I_k&\dots&I_k\\
&I_k&I_k&\dots&I_k\\
&I_k&2\cdot I_k&\dots &(p-1)\cdot I_k
\end{matrix}
\right]{\raisebox{-3ex}{.}}
$$
Note that the rank function is integer valued and does not depend on $k$, so this uniquely defines a matroid for every $p$. For example, for $p=2$ this is a $k$-linear representation of the Fano matroid. Clearly, by construction, $R_p$ is $k$-linearly representable over any field of characteristic $p$ for every $k$. The converse is implied by the result of Reid\footnote{Reid proved the result for linear representations, but almost the same proof works for $k$-linear representations as well.}.
\begin{ttheorem}\label{th:M7lin}
For every $k\in \N$, the matroid $R_p$ is $k$-linearly representable over $\F$ if and only if $\Char(\F)=p$.
\end{ttheorem}
Gordon~\cite{Gordon88} studied the algebraic representability of this class of matroids (they are denoted $M_p$ in \cite{Gordon88}), called the Reid geometries. He proved the following theorem:
\begin{ttheorem}\label{th:M7}
The matroid $R_p$ is algebraically representable over $\F$ if and only if $\Char(\F)=p$.
\end{ttheorem}

\paragraph{The Rank-3 Dowling Group Geometry.}
The linear representability of the Dowling geometries was determined already by Dowling~\cite{Dowlingb}.

\remove{The Dowling group geometries were introduced by Dowling~\cite{Dowlinga,Dowlingb}. The rank-$r$ Dowling geometry of the group $G$ is denoted by $Q_r(G)$. The general construction of $Q_r(G)$ can be found in \cite{Dowlingb} or \cite[Chapter 6.10]{Oxl11}. We use a slightly different construction of the rank-3 Dowling group geometries, following \cite{Aner}. For completeness, we include the definition in Appendix \ref{app:dowling}.

Notice that in the geometric representation of $Q_3(G)$, every point of $Q_3(G)$ lies on (at least) one of the following lines: $\set{p_1,p_2}$,$\set{p_2,p_3}$, $\set{p_1,p_3}$. We refer to these lines as the {\em edges} of $Q_3(G)$.

Dowling showed that the rank $r$ (for $r\geq 3$) Dowling geometry associated with a group $G$ is linearly representable over $\F$ exactly when $G$ is a subgroup of $\F^{\times}$ (the group of invertible elements of $\F$):}
\begin{ttheorem}\label{th:dowlinglin}
For $r\geq 3$, the matroid $Q_r(G)$ is linearly representable over a field $\F$ if and only if $G$ is a subgroup of $\F^{\times}$. In particular, if $Q_r(G)$ is linearly representable then $G$ is cyclic.
\end{ttheorem}

Beimel et al.~\cite{Aner} generalized this statement (for $r=3$) to multi-linear representations.\remove{\footnote{The fact that a $2$-linear representation of the quaternion group exists was already shown by Pendavingh and van Zwam~\cite{Zwam13}.}}
\begin{ttheorem}\label{th:dowlingmul}
For a finite group $G$, the matroid $Q_{3}(G)$ is $k$-linearly representable over a field $\F$ if and only if there is a fixed-point free representation $\rho\colon G\rightarrow \GL_{k}(\F)$. Furthermore, if such a representation $\rho$ exists then the block matrix
\remove{
$$\,\,\,p_1\,\,\,\,\,\, p_2\,\,\,\,\,\,p_3\,\,\,\,\,\,\,\,g^{(1)}_1\,\,\,\,\,\,\,\,\,\,\,\,\,\,\,\,\,\,\,\,\,g^{(1)}_n\,\,\,\,\,\,\,\,g^{(2)}_1\,\,\,\,\,\,\,\,\,\,\,\,\,\,\,\,\,\,\,\,\,g^{(2)}_n\,\,\,\,\,\,\,\,\,\,g^{(3)}_1\,\,\,\,\,\,\,\,\,\,\,\,\,\,\,\,\,\,\,\,\,\,\,g^{(3)}_n$$
\begin{eqnarray*}
\mathbf{A_{\rho}}:=\left[\begin{matrix}
I_{k}&0&0&-I_{k}&\ldots &-I_{k}\\
0&I_{k}&0&\rho (g_{1})&\ldots &\rho (g_{n})\\
0&0&I_{k}&0&\ldots &0
\end{matrix}\right.
\left.\begin{matrix}
0&\ldots &0&\rho (g_{1})&\ldots &\rho (g_{n})\\
-I_{k}&\ldots &-I_{k}&0&\ldots &0\\
\rho (g_{1})&\ldots &\rho (g_{n})&-I_{k}&\ldots &-I_{k}
\end{matrix}
\right]
\end{eqnarray*}
}
\[
  \mathbf{A_{\rho}}:=\kbordermatrix{
    & p_1 & p_2 & p_3 & g^{(1)}_1 & & g^{(1)}_n & g^{(2)}_1 & & g^{(2)}_n & g^{(3)}_1 & & g^{(3)}_n  \\
     & I_{k}&0&0&-I_{k}&\ldots &-I_{k} & 0&\ldots &0&\rho (g_{1})&\ldots &\rho (g_{n}) \\
     & 0&I_{k}&0&\rho (g_{1})&\ldots &\rho (g_{n}) & -I_{k}&\ldots &-I_{k}&0&\ldots &0 \\
     & 0&0&I_{k}&0&\ldots &0 &\rho (g_{1})&\ldots &\rho (g_{n})&-I_{k}&\ldots &-I_{k} \\
  }
\]
is a $k$-linear representation of $Q_{3}(G)$.
\end{ttheorem}

We will also need the following simple, well-known lemma about $Q_3(G)$. 
\begin{lemma}\label{lem:maxlines}
Let $M$ be a matroid on the same ground set as $Q_3(G)$, such that every dependent set of $Q_3(G)$ is dependent in $M$. If $M$ is simple and $r(M)=3$, then $M=Q_3(G)$.
\end{lemma}
The lemma basically says that $Q_3(G)$ is maximally dependent, in the sense that, no new line can be added without causing some degeneracy. For completeness, in Section \ref{sec:proofs} we present a proof that shares some similar ideas with the proof of our main theorem.

\section{A Non-Algebraic Multi-Linearly Representable Matroid}\label{sec:main}
In this Section we prove that there exists a multi-linear, non-algebraic matroid. The matroid is constructed as a direct sum of the Reid Geometry $R_p$, with $p=7$, and a new matroid that we introduce, $Q^{NF}_3(\mathsf{Q}_8)$. The matroid $Q^{NF}_3(\mathsf{Q}_8)$ contains both the non-Fano matroid and the rank-3 Dowling geometry of the quaternions as minors.

To show that $Q^{NF}_3(\mathsf{Q}_8)\oplus R_7$ is not algebraic, we prove that the matroids $R_7$ and $Q^{NF}_3(\mathsf{Q}_8)$ cannot be algebraically representable over the same field. From the result of Gordon~\cite{Gordon88}, we know that $R_7$ is algebraically representable over $\K$ if and only if $\Char(\K)=7$. 
So it remains to show that $Q^{NF}_3(\mathsf{Q}_8)$ is not algebraic over any field $\K$ with $\Char(\K)=7$. We prove a stronger statement: $Q^{NF}_3(\mathsf{Q}_8)$ is not algebraic over any field $\K$ with $\Char(\K)\neq 2$. To prove this, we use derivations. 

We first show that the representation can be chosen ``correctly", in similar manner to the works of Lindstrom~\cite{Lind85} and Gordon~\cite{Gordon88}. We show that the correct representation induces a linear representation of a different matroid,  $\overline{Q^{NF}_3(\mathsf{Q}_8)}$, which retains some of the structure of $Q^{NF}_3(\mathsf{Q}_8)$. The linear representation of $\overline{Q^{NF}_3(\mathsf{Q}_8)}$ is over a field $\E$ that extends $\K$. Therefore, finding restrictions on the characteristics of $\K$ reduces to finding restrictions on the characteristics of $\E$.

Then, we investigate the structure of $\overline{Q^{NF}_3(\mathsf{Q}_8)}$. We show that the degeneracy generated by moving from $Q^{NF}_3(\mathsf{Q}_8)$ to $\overline{Q^{NF}_3(\mathsf{Q}_8)}$ can be described by a congruence relation on $\mathsf{Q}_8$. From this, we deduce that $\overline{Q^{NF}_3(\mathsf{Q}_8)}$ contains the Fano matroid as a minor. Since the Fano matroid is linearly representable only over fields of characteristic 2, it forces the required restriction on $\E$. 

To show that $Q^{NF}_3(\mathsf{Q}_8)\oplus R_7$ is $2$-linearly representable, we give an explicit representation over the field $\F_{49}$.

\paragraph{A new matroid.}
We now present a new matroid. We construct it by giving an explicit $2$-linear representation. As explained, it is necessary to verify that the representation indeed defines a matroid.

Let $\F_{49}$ be the unique field with 49 elements ($\Char(\F_{49})=7$) and let $\zeta_4\in \F_{49}$ be a primitive root of unity of order 4. Denote by $\mathsf{Q}_8=\set{\mathsf{e},\mathsf{-e},\mathsf{i},\mathsf{-i},\mathsf{j},\mathsf{-j},\mathsf{k},\mathsf{-k}}$ the quaternion group, and by $\rho\colon \mathsf{Q}_8\to \GL_2(\F_{49})$ the following representation:\footnote{The quaternion group also has a different fixed-point free representation over $\F_7$ which works out similarly. We use $\F_{49}$ only for simplification of the proof of Lemma \ref{lem:q3nf} in Section \ref{sec:proofs}.}
\begin{eqnarray*}
\rho(\mathsf{\pm e})=\pm \left(\begin{matrix}
 1&0\\0& 1\end{matrix}\right),\rho(\mathsf{\pm i})=\pm\left(\begin{matrix}
 \zeta_4&0\\0& -\zeta_4\end{matrix}\right),\rho(\mathsf{\pm j})=\pm\left(\begin{matrix}
0& -1\\ 1&0\end{matrix}\right),\rho(\mathsf{\pm k})=\pm\left(\begin{matrix}
0&-\zeta_4\\ -\zeta_4&0\end{matrix}\right).
\end{eqnarray*} 
It is easy to see that $\rho$ is fixed-point free. We now look at the matrix
\remove{
$$\,\,\,\,\,\,p_1\,\,\,\,\,\, p_2\,\,\,\,\,p_3\,\,\,\,\,\,\,\,\mathsf{e}^{(1)}\,\,\,\,\,\,\,(-\mathsf{e})^{(1)}\,\,\,\,\,\,\,\,\,\,\,(-\mathsf{k})^{(1)}\,\,\,\mathsf{e}^{(2)}\,\,\,\,\,\,\,\,\,\,\,\,\,\,\,(-\mathsf{k})^{(2)}\,\,\,\,\mathsf{e}^{(3)}\,\,\,\,\,\,\,\,\,\,\,\,\,\,\,\,\,(-\mathsf{k})^{(3)}\,\,\,\,\,O$$
\begin{eqnarray*}
\mathbf{A^*_{\rho}}:=\left[\begin{matrix}
I_{2}&0&0&-I_{2}&-I_2&\ldots &-I_{2}\\
0&I_{2}&0&\rho (\mathsf{e})&\rho(-\mathsf{e})&\ldots &\rho (-\mathsf{k})\\
0&0&I_{2}&0&0&\ldots &0
\end{matrix}\right.
\left.\begin{matrix}
0&\ldots &0&\rho (\mathsf{e})&\ldots &\rho (-\mathsf{k})&I_2\\
-I_{2}&\ldots &-I_{2}&0&\ldots &0&I_2\\
\rho (\mathsf{e})&\ldots &\rho (-\mathsf{k})&-I_{2}&\ldots &-I_{2}&I_2
\end{matrix}
\right]{\raisebox{-3ex}{.}}
\end{eqnarray*}
}
\[
  \mathbf{A^*_{\rho}}:=\kbordermatrix{
    & p_1 & p_2 & p_3 & \mathsf{e}^{(1)} & -\mathsf{e}^{(1)} & & -\mathsf{k}^{(1)} & \mathsf{e}^{(2)} & & -\mathsf{k}^{(2)} & \mathsf{e}^{(3)} & & -\mathsf{k}^{(3)} & O \\
     & I_{2}&0&0&-I_{2}&-I_2&\ldots &-I_{2} & 0&\ldots &0&\rho (\mathsf{e})&\ldots &\rho (-\mathsf{k})&I_2\\
     & 0&I_{2}&0&\rho (\mathsf{e})&\rho(-\mathsf{e})&\ldots &\rho (-\mathsf{k}) &-I_{2}&\ldots &-I_{2}&0&\ldots &0&I_2\\
     & 0&0&I_{2}&0&0&\ldots &0 &\rho (\mathsf{e})&\ldots &\rho (-\mathsf{k})&-I_{2}&\ldots &-I_{2}&I_2\\
  }{\raisebox{-3ex}{.}}
\]
\begin{lemma}\label{lem:q3nf}
The matrix $\mathbf{A^*_{\rho}}$ is a $2$-linear representation over the field $\F_{49}$, of a matroid which contains both $Q_3(\mathsf{Q}_8)$ and the non-Fano matroid\remove{$F^-_7$} as minors.
\end{lemma}

We verify that this is indeed a 2-linear representation, by verifying that the rank function of the matroid is properly defined , i.e., the rank of every relevant sub-matrix is even. Since the proof is simple and technical, we postpone it until Section \ref{sec:proofs}. We note that for a specific $k$-linear representation (e.g., this one), this can also be verified by a computer.

The restriction to $\set{p_1,p_2,p_3,\mathsf{e}^{(1)},...,(-\mathsf{k})^{(3)}}$ is $Q_3(\mathsf{Q}_8)$, and the restriction to $\{p_1$, $p_2$, $p_3$, $(\mathsf{-e})^{(1)}$, $(\mathsf{-e})^{(2)}$, $(\mathsf{-e})^{(3)}$, $O\}$ is the non-Fano matroid (cf. Section \ref{sec:proofs}).
\remove{
A different way of constructing the above matroid is the following: take the matroid $Q_3(\mathsf{Q}_8)$ and add a point $O$ such that
\begin{enumerate}
\item $r(\set{p_i,p_j,O})=3$ for all $1\leq i<j \leq 3$,
\item $r(\set{p_i,O,-\mathsf{e}^{(j)}})$, whenever $-\mathsf{e}^{(j)}$ is on the edge opposite of $p_i$.
\end{enumerate}
}

Denote the above matroid by $Q^{NF}_3(\mathsf{Q}_8)$, and let us present our main theorem. We show that if $Q^{NF}_3(\mathsf{Q}_8)$ is algebraic, then its algebraic representation must be over a field of characteristic 2.
\begin{ttheorem}\label{th:main}
If $Q^{NF}_3(\mathsf{Q}_8)$ is algebraically representable over $\K$, then $\Char(\K)=2$.
\end{ttheorem}

The proof is based on studying the matroid induced by the derivations. By Piff's conjecture, we may consider only representations over prime fields. There are 2 main lemmas in the proof. In Lemma \ref{lem:replace} we follow the strategy of Lindstr\" om \cite{Lind85} and Gordon \cite{Gordon88}, and show that if an algebraic representation of $Q^{NF}_3(\mathsf{Q}_8)$ exists, then after replacing elements by their powers we get another representation, which is separable over $\F_p$, and certain derivations do not vanish. In Lemma \ref{lem:Fanocontain} we show that the induced derivations matroid of the representation from Lemma \ref{lem:replace} contains the Fano matroid as a minor. \remove{define an equivalence relation on the points of $Q^{NF}_3(\mathsf{Q}_8)$, and show that it induces a congruence relation on the elements of $\mathsf{Q}_8$. This forces the linearly represented matroid of the derivations, $\overline{Q^{NF}_3(\mathsf{Q}_8)}$, to contain the Fano matroid as a minor.\remove{ (while $Q^{NF}_3(\mathsf{Q}_8)$ contains the non-Fano matroid as a minor).} 
From this, we obtain Theorem \ref{th:main} as a corollary.\remove{ We stress that for the proof we are assuming that an algebraic representation of $Q^{NF}_3(\mathsf{Q}_8)$ exists; as far as we know it is possible that this is not the case (and then the theorem is trivially correct).}}
\begin{proof}
First note that $Q^{NF}_3(\mathsf{Q}_8)$ is not linear. This follows from Theorem \ref{th:dowlinglin}, because $Q_3(\mathsf{Q}_8)$ is a sub-matroid by restriction, and $\mathsf{Q}_8$ is not cyclic. Therefore, $Q^{NF}_3(\mathsf{Q}_8)$ does not admit an algebraic representation over a field of characteristic 0.

Assume that $Q^{NF}_3(\mathsf{Q}_8)$ admits an algebraic representation $\pi\colon E\to \F:=\F_p(\pi(E))$ for some prime $p$.
Set $z:=\pi(O)$, $x_i:=\pi(p_i)$, and $y_{g,i}:=\pi(g^{(i)}$) for every $1\leq i \leq 3$ and $g\in \mathsf{Q}_8$. Since $\set{p_1,p_2,p_3}$ is a base, the elements $\set{x_1,x_2,x_3}$ are algebraically independent and every $y_{g,i}$ is algebraically dependent on $\set{x_1,x_2,x_3}$. \remove{ Therefore, $\F_p(\pi(E))\subset \overline{\F_p(x_1,x_2,x_3)}$, where $\overline{\F_p(x_1,x_2,x_3)}$ denotes the algebraic closure of $\F_p(x_1,x_2,x_3)$.} However, $\F/\F_p$ is not necessarily a separable extension.
 
So in the following lemma, we replace $\pi$ with a different representation $\pi'$, where the $x_is,y_{g,i}s$, and $z$ are replaced by their powers, which will be powers of $p$, i.e., $\pi'(p_i)=x'_i=x_i^{\epsilon_i},\pi'(g^{(i)})= y '_{g,i}=y_{g,i}^{\epsilon_{g,i}},\pi'(O)=z'=z^{\epsilon_z}$ with $\epsilon_i,\epsilon_{g,i},\epsilon_z\in \set{p^m|m\in \Z}$. By Lemma \ref{lem:basicreplace}, $\pi'$ is also a representation of $Q^{NF}_3(\mathsf{Q}_8)$ over $\F_p$.

For every element $s\in \E$, with $\E$ a separable extension of $\F_p(x'_1,x'_2,x'_3)$, denote the gradient $\nabla(s)=(\bar{D}_1(s),\bar{D}_2(s),\bar{D}_3(s))$, with $\bar{D}_i(x'_j)$ being the extension of $\frac{\partial}{\partial x'_i}$ to $\E$. \remove{$\bar{D}_i(x'_j)=\delta_{i,j}$ for $1\leq i,j\leq 3$.}
\begin{lemma}\label{lem:replace}
There exists an algebraic repesentation $\pi'\colon E\to\E$ as above such that \remove{If $\pi$ is an algebraic representation of $Q^{NF}_3(\mathsf{Q}_8)$, then there exists a representation $\pi'$ of $Q^{NF}_3(\mathsf{Q}_8)$, over the same field $\F_p$ as $\pi$, such that}
\begin{enumerate}
\item The extension $\F_p(x'_1,x'_2,x'_3)\subseteq \E=\F_p(\pi'(E))$ is separable,
\item For every element $e\in E$ we have $\nabla(\pi'(e))\neq 0$, 
\item No 3 vectors of the set $\set{\nabla(x'_1),\nabla(x'_2),\nabla(x'_3),\nabla(z')}$ are linearly dependent.
\end{enumerate}
\end{lemma}
\begin{proof}
The proof is based on the following 3 claims:\remove{. The first, which is a result of Lindstr\"om, deals with the elements $\set{p_1,p_2,p_3,(\mathsf{-e})^{(1)},(\mathsf{-e})^{(2)},(\mathsf{-e})^{(3)},O}$, and the other 2 claims deal with the rest of the elements.}
\begin{claim}[Lindstr\" om\cite{Lind85} ]
There exists $\pi'$ as above such that
\remove{We can find a mapping of $Q^{NF}_3(\mathsf{Q}_8)$, by replacing $x_1,x_2,x_3,y_{\mathsf{-e},1},y_{\mathsf{-e},2},y_{\mathsf{-e},3},z$ with their powers $x'_1,x'_2,x'_3,y'_{\mathsf{-e},1},y'_{\mathsf{-e},2},y'_{\mathsf{-e},3},z'$,  such that}
\begin{enumerate}
\item
$y'_{\mathsf{-e},1},y'_{\mathsf{-e},2},y'_{\mathsf{-e},3},z'$ are separable over $\F_p(x'_1,x'_2,x'_3)$,
\item Every subset of $\set{\nabla(x'_1),\nabla(x'_2),\nabla(x'_3),\nabla(y'_{\mathsf{-e},1}),\nabla(y'_{\mathsf{-e},2}),\nabla(y'_{\mathsf{-e},3}),\nabla(z')}$ of size 2 is linearly independent and the same holds for every subset of size 3 of  $\{\nabla(x'_1)$, $\nabla(x'_2)$, $\nabla(x'_3)$, $\nabla(z')\}$.
\end{enumerate}
\end{claim}
This is implied by Lindstr\" om's result since the restriction to $\{p_1$, $p_2$, $p_3$, $(\mathsf{-e})^{(1)}$, $(\mathsf{-e})^{(2)}$, $(\mathsf{-e})^{(3)}$, $O\}$ is the non-Fano matroid.
\remove{This follows from Lindstr\" om \cite{Lind85}, since the restriction to $\{p_1$, $p_2$, $p_3$, $(\mathsf{-e})^{(1)}$, $(\mathsf{-e})^{(2)}$, $(\mathsf{-e})^{(3)}$, $O\}$ is the non-Fano matroid.}
From now on, we fix the values of $x'_1,x'_2,x'_3,y'_{\mathsf{-e},1},y'_{\mathsf{-e},2},y'_{\mathsf{-e},3},z'$ and modify only the rest of the values of $\pi'$. 
Fix $g$ and denote $y:=y_{g,1}$.
\remove{
So we now continue on the rest of the elements $y_{g,1}$ for $g\in \mathsf{Q_8}\setminus \set{\mathsf{-e}}$. Fix $g$ and denote $y:=y_{g,1}$. }
\begin{claim}\label{lem:sep}
By replacing $y$ with $y':=y^{p^m}$ for some $m\in \Z$, we may assume $y'$ is separable over $\F_p(x'_1,x'_2,x'_3)$.
\end{claim}
\begin{proof}
The set $\set{x'_1,y,x'_2}$ is algebraically dependent while $\set{x'_1,x'_2}$ is algebraically independent. Denote by $m_{y}$ the minimal polynomial of $y$ in $\F_p(x'_1,x'_2)[X]$. If $y$ is not separable over $\F_p(x'_1,x'_2)$ then the formal derivative is the zero polynomial (i.e., $(m_{y})'=0$), and there exists an $m\in \N$ such that $m_{y}(X)=f(X^{p^m})$, where $f'\neq 0$. After replacing $y$ with $y':=(y)^{p^m}$ we\remove{ now have that for the minimal polynomial $(m_{y'})'\neq 0$, so} may assume that $y'$ is separable algebraic over $\F_p(x'_1,x'_2)$, and, therefore, over $\F_p(x'_1,x'_2,x'_3)$.
\end{proof}
So, $\nabla(y')$ is now well defined. However, we further need to show (by possibly replacing $y'$ again) that $\nabla(y')\neq 0$.
\begin{claim}
By replacing $y$ with $y':=y^{p^m}$ for some $m\in \Z$, we may assume $y'$ is separable over $\F_p(x'_1,x'_2,x'_3)$ and $\nabla(y')\neq 0$.
\end{claim}
\begin{proof}
\remove{Following Claim \ref{lem:sep}, we may assume that $y$ is already separable over $\F_p(x'_1,x'_2,x'_3)$.} If $\bar{D}_1(y)\neq 0$ or $\bar{D}_2(y)\neq 0$ then we are done. Otherwise, $m_y$, the minimal polynomial of $y$ in $\F_p(x'_1,x'_2,x'_3)[X]$, is a monic polynomial with coefficients in $\F_p(x'_1,x'_2)$, i.e., of the form $m_y=\Sigma_{i=0}^{{\ell}}a_iX^i$ with $a_i\in \F_p(x'_1,x'_2)$ and $a_{\ell}=1$.

Now, for any derivation $\bar{D}(y)=-\frac{m_y^D(y)}{(m_y)'(y)}$ by Lemma \ref{cor:der2}.
We see that $\bar{D}_1(y)=\bar{D}_2(y)=0$ implies $(m_y)^{D_1}(y)=(m_y)^{D_2}(y)=0$. This implies that $(m_y)^{D_1}=(m_y)^{D_2}=0$, because $(m_y)^{D}$ is of smaller degree than $m_y$ (since $D(a_{\ell})=D(1)=0$). So for all $1\leq i\leq {\ell}$, $D_1(a_i)=D_2(a_i)=0$. Since $D_1,D_2$ are a basis of $\Der_{\F_p}(\F_p(x'_1,x'_2))$, this means that $D(a_i)=0$ for any derivation $D\in \Der_{\F_p}(\F_p(x'_1,x'_2))$.

So we can rewrite the coefficients $a_i=(\tilde{a}_i)^{p^{\gamma_i}}$ with $\tilde{a}_i\in \F_p(x'_1,x'_2)$ and $\gamma_i\in \N$. Let $\gamma^*_i$ be the maximal such $\gamma_i$ for each $i$ (take $\gamma^*_i=\infty$ if $a_i\in\F_p$), and take $\gamma=\min(\gamma^*_{i})$. ($\gamma<\infty$ otherwise $m_y(Y)\in \F_p[Y]$ and $\deg_{tr}(\F_p(y)/\F_p)=0$).

Since 
for any $a,b\in\F_p(x'_1,x'_2)$ we have $a^p+b^p=(a+b)^p$, we see that $m_y(X^{p^{\gamma}})=(f(X))^{p^{\gamma}}$ for some $f\in \F_p(x'_1,x'_2)[X]$.
So by replacing $y$ with $y'=y^{p^{-\gamma}}$, we now have that $f$ 
is its minimal polynomial, $f'\neq 0$, and by the choice of $\gamma$ at least one of $D_1(y'),D_2(y')$ is not zero, so we are done. Therefore, $y'$ is the desired $y'_{g,1}$.
\end{proof}

By the same argument, we can properly replace $y_{g,2}$ and $y_{g,3}$ for every $g\in \mathsf{Q_8}\setminus \set{\mathsf{-e}}$, which completes the proof of Lemma \ref{lem:replace}.\remove{, and we are done with Part 1.}
\end{proof}

Let $\pi'$ be as in Lemma \ref{lem:replace}. Consider the matroid $\overline{Q^{NF}_3(\mathsf{Q}_8)}$, induced by the derivations of $\pi'$.
\remove{We now look at $\overline{Q^{NF}_3(\mathsf{Q}_8)}$, the matroid induced by the derivations of $\pi'$, the algebraic representation of $Q^{NF}_3(\mathsf{Q}_8)$ from Lemma \ref{lem:replace}.}
\begin{lemma}\label{lem:Fanocontain}
The restriction of the matroid $\overline{Q^{NF}_3(\mathsf{Q}_8)}$ to $\set{p_1,p_2,p_3,(\mathsf{-e})^{(1)},(\mathsf{-e})^{(2)},(\mathsf{-e})^{(3)},O}$ is the Fano matroid.
\end{lemma}
\begin{proof} We break the proof of this lemma into a series of claims that analyze the structure of $\overline{Q^{NF}_3(\mathsf{Q}_8)}$.
As we saw in Section \ref{sec:derivations}, any dependent set in $Q^{NF}_3(\mathsf{Q}_8)$ must also be dependent in $\overline{Q^{NF}_3(\mathsf{Q}_8)}$. However, new dependent sets can arise. Since $Q^{NF}_3(\mathsf{Q}_8)$ is of rank 3, only 3 types of new dependencies can occur:
\begin{enumerate}
\item A singleton $\set{y}$ becomes dependent,
\item A set of two non-parallel elements $\set{y_1,y_2}$ becomes dependent,
\item An independent set of three elements $\set{y_1,y_2,y_3}$ becomes dependent.
\end{enumerate}
For simplicity, we say that in type 1 the point $y$ {\em vanishes}, in type 2 the points $y_1,y_2$ {\em unite}, and in type three the {\em line} $\set{y_1,y_2,y_3}$ is added. If $y_1,y_2$ unite then for every $x$ we will say that the lines $\set{x,y_1}$ and $\set{x,y_2}$ unite (since they become the same line in $\overline{Q^{NF}_3(\mathsf{Q}_8)}$). \remove{Geometrically, this is how we ``actually see" this, e.g.,} In Example \ref{ex:alg2lin} the point $2$ vanished, the points $1,4$ united, and the line $\set{1,3,6}$ was added. We can also see that since $1,4$ united, the lines $\set{1,5,3}$ and $\set{4,6,3}$ united so $3,5$ and $6$ are collinear in $\overline{M}$.

Now consider the geometric representation of $\overline{Q^{NF}_3(\mathsf{Q}_8)}$. The following observation follows directly from Lemma \ref{lem:replace}.\remove{, the matroid induced by the derivations of $\pi'$, the algebraic representation of $Q^{NF}_3(\mathsf{Q}_8)$ from Lemma \ref{lem:replace}.}
\begin{observation}\label{ob:nonvanish}
None of the points vanish. Furthermore,\remove{ by the result of Lindstr\" om,} no 3 points of the set $\set{p_1,p_2,p_3,O}$ are collinear in $\overline{Q^{NF}_3(\mathsf{Q}_8)}$.
\end{observation}
\remove{This follows directly from Lemma \ref{lem:replace}.} We next make a key point - the matroid $\overline{Q^{NF}_3(\mathsf{Q}_8)}$ is not simple.

\begin{claim}\label{lem:unite}
Some of the points of $Q^{NF}_3(\mathsf{Q}_8)$ unite in $\overline{Q^{NF}_3(\mathsf{Q}_8)}$.
\end{claim}
\begin{proof}
Assume that no points unite. Recall that $\overline{Q^{NF}_3(\mathsf{Q}_8)}$ is linearly representable. Since $Q^{NF}_3(\mathsf{Q}_8)\setminus O= Q_3(\mathsf{Q}_8)$, it follows from Lemma \ref{lem:maxlines} that no line can be added, other than, possibly, lines passing through $O$. But this implies that $\overline{Q^{NF}_3(\mathsf{Q}_8)}\setminus O=Q^{NF}_3(\mathsf{Q}_8)\setminus O= Q_3(\mathsf{Q}_8)$ is linearly representable, contradicting Theorem \ref{th:dowlinglin}. 
\end{proof}
So some points necessarily unite. We will write $a\sim b$ if the elements $a,b$ of $Q^{NF}_3(\mathsf{Q}_8)$ unite.
This is clearly an equivalence relation, so we denote the representative of a point $a$ in $\overline{Q^{NF}_3(\mathsf{Q}_8)}$  by $\overline{a}$.

We now show some restrictions on the points that can unite.\remove{ First, we show that if 2 distinct points unite, then neither of the points is $O$, nor any of the joints.}
\begin{claim}\label{lem:O}
The point $O$ does not unite with any other point.
\end{claim}
\begin{proof}
Suppose $O$ unites with a point $a$. Then $a$ lies on one of the edges \remove{(i.e., the edges of $Q_3(\mathsf{Q}_8)=Q^{NF}_3(\mathsf{Q}_8)\setminus O$)}, say $\set{p_1,p_2}$. So the set $\set{p_1,O,p_2}$ is dependent in $\overline{Q^{NF}_3(\mathsf{Q}_8)}$, contradicting Observation \ref{ob:nonvanish}.
\end{proof}

\begin{claim}\label{lem:joints}
A joint cannot unite with any other point.
\end{claim}
\begin{proof}
From Observation \ref{ob:nonvanish}\remove{ and claim \ref{lem:O}}, the joints cannot unite with each other or with $O$. So a joint might either unite with a point on an adjacent edge or with a point on the opposite edge. However, both options are impossible:
\begin{itemize}
\item Assume, without loss of generality, that $p_1$ unites with $g^{(1)}$. So the following 3 lines unite: $\ell_1=\set{p_1,O,(-e)^{(2)}}$, $\ell_3=\set{g^{(1)},(-e)^{(2)},((-eg)^{-1})^{(3)}}$, and $\ell_3=\set{p_1,((-eg)^{-1})^{(3)},p_3}$. So $\overline{p_1},\overline{O},\overline{p_3}$ lie on the same line in $\overline{Q^{NF}_3(\mathsf{Q}_8)}$, contradicting Observation \ref{ob:nonvanish}.
\item Assume, without loss of generality, that $p_1$ unites with $g^{(2)}$. So $\set{p_2,p_1,p_3}$ is dependent in $\overline{Q^{NF}_3(\mathsf{Q}_8)}$, again contradicting Observation \ref{ob:nonvanish}.\qedhere
\end{itemize}
\end{proof}

\remove{We next show that if 2 points unite then they lie on the same edge in $Q^{NF}_3(\mathsf{Q}_8)$.}
\begin{claim}
If $i\neq j$, then for every $g,h\in \mathsf{Q}_8$, we have $g^{(i)}\nsim h^{(j)}$.
\end{claim}
\begin{proof}
Assume, without loss of generality, that $g^{(1)}\sim h^{(2)}$. So the 2 lines $\{p_1,g^{(1)},p_2\}$ and $\{p_2,h^{(2)},p_3\}$ unite. Therefore,  $\set{p_1,p_2,p_3}$ are collinear in $\overline{Q^{NF}_3(\mathsf{Q}_8)}$, contradicting Observation \ref{ob:nonvanish}.
\end{proof}
So we now know that if 2 points unite with each other, then they are between the same joints. We proceed to show that this unification is symmetric.
\begin{claim}
Let $g,h\in \mathsf{Q}_8$. If $g^{(i)}\sim h^{(i)}$ for some $i\in \set{1,2,3}$, then $(g^{-1})^{(\ell)}\sim (h^{-1})^{(\ell)}$ and $g^{(\ell)}\sim h^{(\ell)}$ for all $\ell\in \set{1,2,3}$.
\end{claim}
\begin{proof}
Assume, without loss of generality, that $g^{(1)}\sim h^{(1)}$. So there is a line in $\overline{Q^{NF}_3(\mathsf{Q}_8)}$ passing through the 4 points $\{\overline{g^{(1)}},$ $\overline{(g^{-1})^{(2)}},$ $\overline{(h^{-1})^{(2)}},$ $\overline{e^{(3)}} \}$, since the line $\{g^{(1)},(g^{-1})^{(2)},e^{(3)}\}$ unites with the line $\{h^{(1)},(h^{-1})^{(2)},e^{(3)}\}$. If we assume that $(g^{-1})^{(2)}\nsim (h^{-1})^{(2)}$, then these 4 points are distinct. But since $\{\overline{p_2},\overline{(g^{-1})^{(2)}},\overline{(h^{-1})^{(2)}},\overline{p_3} \}$ are also on a line, the two lines must be the same line (because $\overline{(g^{-1})^{(2)}}$ and $\overline{(h^{-1})^{(2)}}$ are 2 distinct points on both lines). This line must also be the same as $\{\overline{p_2},\overline{g^{(1)}},\overline{p_1}\}$ (now $\overline{p_2},\overline{g^{(1)}}$ are 2 distinct points on both lines), so $\{\overline{p_1},\overline{p_2},\overline{p_3}\}$ are collinear, a contradiction. So $(g^{-1})^{(2)}\sim (h^{-1})^{(2)}$, and now by symmetric arguments $g^{(3)}\sim h^{(3)}$. By further applying these arguments symmetrically,  $(g^{-1})^{(\ell)}\sim (h^{-1})^{(\ell)}$ and $g^{(\ell)}\sim h^{(\ell)}$ for all $\ell\in \set{1,2,3}$.
\end{proof}
So we can now safely write $g\sim h$ for $g,h\in \mathsf{Q}_8$, and the meaning is clear. This induces an equivalence relation on $\mathsf{Q}_8$. We next show that it is also a congruence relation.\remove{, i.e., if $g\sim g'$ and $h\sim h'$, then $gh\sim g'h'$. }
\begin{claim}\label{lem:cong}
For every $g,g',h,h'\in \mathsf{Q}_8$, if $g\sim g'$ and $h\sim h'$, then $gh\sim g'h'$.
\end{claim}
\begin{proof}
Assume $g\sim g'$ and $h\sim h'$. From the dependencies of $Q^{NF}_3(\mathsf{Q}_8)$, we have 2 lines, $\{h^{(1)},$ $g^{(2)},$ $((gh)^{-1})^{(3)}\}$ and $\set{(h')^{(1)},(g')^{(2)},((g'h')^{-1})^{(3)}}$, which in $\overline{Q^{NF}_3(\mathsf{Q}_8)}$ must be the same line (they share 2 points since  $\overline{g^{(1)}}=\overline{g'^{(1)}}$ and $\overline{h^{(2)}}=\overline{h'^{(2)}}$ ). If $(gh)^{-1}\nsim (g'h')^{-1}$ then this line must also be the same line as $\set{\overline{p_1},\overline{((gh)^{-1})^{(3)}},\overline{((g'h')^{-1})^{(3)}},\overline{p_3}}$; therefore, the same line as $\set{\overline{p_1},\overline{h^{(1)}},\overline{p_2}}$. So $\set{\overline{p_1},\overline{p_2},\overline{p_3}}$ are on the same line, a contradiction. Therefore, $(gh)^{-1}\sim (g'h')^{-1}$, implying $gh \sim g'h'$.
\end{proof}
We know from the previous claims that there exist $h_1,h_2\in \mathsf{Q}_8$, with $h_1\neq h_2$, such that $h_1\sim h_2$. Therefore, the kernel of the congruence relation, i.e., the equivalence class of $\mathsf{e}$, is non-trivial. Thus, it must contain $\mathsf{-e}$, as $\mathsf{-e}$ is the square of any element of $\mathsf{Q}_8$ other than $\mathsf{e},-\mathsf{e}$.

We are now ready for the punch. We show that the Fano matroid is a restriction of $\overline{Q^{NF}_3(\mathsf{Q}_8)}$: 
we have shown that $\mathsf{e}\sim -\mathsf{e}$, so $\mathsf{e}^{(\ell)},(-\mathsf{e})^{(\ell)}$ unite for $\ell=1,2,3$. Since $\set{\mathsf{e}^{(1)},\mathsf{e}^{(2)},\mathsf{e}^{(3)}}$ are collinear in $Q^{NF}_3(\mathsf{Q}_8)$, we have that $\{\mathsf{(-e)}^{(1)}$, $\mathsf{(-e)}^{(2)}$, $\mathsf{(-e)}^{(3)}\}$ are collinear in $\overline{Q^{NF}_3(\mathsf{Q}_8)}$. \remove{Since the restriction of $Q^{NF}_3(\mathsf{Q}_8)$ to $\set{p_1,p_2,p_3,\mathsf{(-e)}^{(1)},\mathsf{(-e)}^{(2)},\mathsf{(-e)}^{(3)},O}$ is the non-Fano matroid, and none of these points vanished or united with each other, }Combining with Observation \ref{ob:nonvanish} and Claims \ref{lem:O} and \ref{lem:joints}, it follows that the restriction of $\overline{Q^{NF}_3(\mathsf{Q}_8)}$ to $\set{p_1,p_2,p_3,\mathsf{(-e)}^{(1)},\mathsf{(-e)}^{(2)},\mathsf{(-e)}^{(3)},O}$ is the Fano matroid. This concludes the proof of Lemma \ref{lem:Fanocontain}.
\end{proof}

Since $\overline{Q^{NF}_3(\mathsf{Q}_8)}$ is linearly representable over $\E$, so is the Fano matroid, which is its minor. Therefore, $\Char(\E)=2$, since the Fano matroid is linearly representable only over fields of characteristic 2 (cf. ~\cite[Proposition 6.4.8]{Oxl11}). Since $\F_p\subset \E$ is a field extension, it follows that $p=2$, which completes the proof of Theorem \ref{th:main}.
\end{proof}

\begin{remark}
In fact, we can now even strengthen Lemma \ref{lem:Fanocontain}: Claim \ref{lem:cong} implies that the equivalence class  $N$ of $\mathsf{e}$ is a normal sub-group of $\mathsf{Q}_8$, and that $\si(\overline{Q^{NF}_3(\mathsf{Q}_8)}\setminus O)$, the simple matroid associated with  $\overline{Q^{NF}_3(\mathsf{Q}_8)}\setminus O$ (this means we look at each equivalence class of $\overline{Q^{NF}_3(\mathsf{Q}_8)}\setminus O$ as unique point in $\si(\overline{Q^{NF}_3(\mathsf{Q}_8)}\setminus O)$, see \cite[Chapter 1.7]{Oxl11}), is $Q_3(\mathsf{Q}_8/N)$. Therefore, $N=\mathsf{Q}_8$, because no quotient $\mathsf{Q}_8/N$, other than the trivial group, is a subgroup of the multiplicative group of a field of characteristic 2. So $\si(\overline{Q^{NF}_3(\mathsf{Q}_8)})$ is itself the Fano matroid.
\end{remark}

\begin{corollary}
The matroid $Q^{NF}_3(\mathsf{Q}_8)\oplus R_7$ is a non-algebraic, $2$-linearly representable matroid.
\end{corollary}
\begin{proof}
Since both $Q^{NF}_3(\mathsf{Q}_8)$ and $R_7$ are $2$-linearly representable over $\F_{7^2}$ by Lemmas \ref{lem:q3nf} and \ref{th:M7lin}, so is their direct sum. On the other hand, if $Q^{NF}_3(\mathsf{Q}_8)\oplus R_7$ is algebraically representable over a field $\K$ so are $Q^{NF}_3(\mathsf{Q}_8)$ and $R_7$ (algebraic representations are minor closed). Therefore, by Theorem \ref{th:main} $\Char(\K)=2$, but by Theorem \ref{th:M7} $\Char(\K)=7$, which is a contradiction.
\end{proof}
Since every multi-linear matroid is secret-sharing, we obtain as an immediate corollary the following important statement about secret-sharing matroids.
\begin{corollary}
There exists a secret-sharing matroid that is not algebraic.
\end{corollary}

\begin{remark}
We note that $Q^{NF}_3(\mathsf{Q}_8)\oplus R_7$ is a non-connected matroid of rank 6 with 45 elements. However, it is possible to find a multi-linear representation of a connected matroid of rank 3 with 38 elements that also contains $Q^{NF}_3(\mathsf{Q}_8)$ and $R_7$ by restriction. The matroid is constructed by connecting the representations in the natural way, and then deleting parallel elements. An interesting question is what is the smallest non-algebraic, multi-linearly representable matroid.
\end{remark}

\section{Proofs of Lemmas \ref{lem:maxlines} and \ref{lem:q3nf}}\label{sec:proofs}
In this section we give the postponed technical proofs of Lemma \ref{lem:maxlines} and Lemma \ref{lem:q3nf}.

\remove{
\begingroup
\def\thelemma{\ref{lem:maxlines}}
\begin{lemma}
Let $M$ be a matroid on the same ground set as $Q_3(G)$, such that every dependent set of $Q_3(G)$ is dependent in $M$. If $M$ is simple and $r(M)=3$, then $M=Q_3(G)$.
\end{lemma}
\addtocounter{lemma}{-1}
\endgroup

\begin{lemma}[Lemma \ref{lem:maxlines}]
Let $M$ be a matroid on the same ground set as $Q_3(G)$, such that every dependent set of $Q_3(G)$ is dependent in $M$. If $M$ is simple and $r(M)=3$, then $M=Q_3(G)$.
\end{lemma}}
\begin{proof}[Proof of Lemma \ref{lem:maxlines}]
\noindent We repeatedly use the fact that any 2 elements of $M$ can only lie on a single line, i.e., if $\{a,b,x_1$, $\dots,$ $x_n\}$ are collinear and $\set{a,b,y_1,\dots,y_m}$ are collinear then $\set{a,b,x_1,\dots,x_n,y_1,\dots,y_m}$ are all on the same line. Recall that in $Q_3(G)$ every point lies on one of the lines $\set{p_1,p_2}$,$\set{p_2,p_3}$, $\set{p_1,p_3}$, which we refer to as the {\em edges} of $Q_3(G)$.

Suppose by contradiction that $M$ contains an additional 3-element dependent set $\set{a,b,c}$. There are 2 cases to consider:

\begin{enumerate}
\item Two of the points, say $a$ and $b$, are on of the same edge, w.l.o.g $\set{p_1,p_2}$. Therefore, $c$ is on a different edge (since $\set{a,b,c}$ was independent in $Q_3(G)$), w.l.o.g $\set{p_1,p_3}$. So in $M$, the line $\set{a,b}=\set{p_1,p_2}$ contains $c$ and, thus, also $p_3$. But in $Q_3$ all the elements lie on one of the lines $\set{p_1,p_2}$, $\set{p_1,p_3}$, $\set{p_2,p_3}$. Therefore, this holds also in $M$. It follows that all the elements of $M$ lie on the same line, so $r(M)=2$, a contradiction.

\item Each of the points is on a different edge. W.l.o.g., $a$ is on the line $\set{p_1,p_2}$, $b$ is on $\set{p_2,p_3}$, and $c$ is on $\set{p_1,p_3}$. So denote $a=g^{(1)},b=h^{(2)},c=k^{(3)}$ with $g,h,k\in G$. Thus, $\set{a,b,((hg)^{-1})^{(3)}}$ is a dependent set in $Q_3(G)$, and hence in $M$. Since $\set{a,b,c}$ was not dependent in $Q_3(G)$, we have that $hgk\neq e\Rightarrow k\neq (hg)^{-1}$. Therefore, $\set{a,b,c}$ and $\set{a,b,((hg)^{-1})^{(3)}}$ are both dependent in $M$, which forces  $\set{a,b,c,((hg)^{-1})^{(3)}}$ to be on the same line. This line contains also $p_1$ and $p_3$ (since $c$ and $((hg)^{-1})^{(3)}$ lie on the line $\set{p_1,p_3}$). Thus, this line contains $p_2$ as well (since $\set{p_1,a,p_2}$ is dependent). So we get a contradiction as in case 1.\qedhere
\end{enumerate}
\end{proof}

\noindent We next prove Lemma \ref{lem:q3nf}. We first recall that $\mathbf{A^*_{\rho}}$ is the block matrix
\[
  \kbordermatrix{
    & p_1 & p_2 & p_3 & \mathsf{e}^{(1)} & -\mathsf{e}^{(1)} & & -\mathsf{k}^{(1)} & \mathsf{e}^{(2)} & & -\mathsf{k}^{(2)} & \mathsf{e}^{(3)} & & -\mathsf{k}^{(3)} & O \\
     & I_{2}&0&0&-I_{2}&-I_2&\ldots &-I_{2} & 0&\ldots &0&\rho (\mathsf{e})&\ldots &\rho (-\mathsf{k})&I_2\\
     & 0&I_{2}&0&\rho (\mathsf{e})&\rho(-\mathsf{e})&\ldots &\rho (-\mathsf{k}) &-I_{2}&\ldots &-I_{2}&0&\ldots &0&I_2\\
     & 0&0&I_{2}&0&0&\ldots &0 &\rho (\mathsf{e})&\ldots &\rho (-\mathsf{k})&-I_{2}&\ldots &-I_{2}&I_2\\
  }{\raisebox{-3ex}{,}}
\]
where $\rho\colon \mathsf{Q}_8\to \GL_2(F_{7^2})$ is the following representation
\begin{eqnarray*}
\rho(\mathsf{\pm e})=\pm \left(\begin{matrix}
 1&0\\0& 1\end{matrix}\right),\rho(\mathsf{\pm i})=\pm\left(\begin{matrix}
 \zeta_4&0\\0& -\zeta_4\end{matrix}\right),\rho(\mathsf{\pm j})=\pm\left(\begin{matrix}
0& -1\\ 1&0\end{matrix}\right),\rho(\mathsf{\pm k})=\pm\left(\begin{matrix}
0&-\zeta_4\\ -\zeta_4&0\end{matrix}\right),
\end{eqnarray*} 
and $\zeta_4$ denotes the primitive root of unity of order 4.

\remove{
\begingroup
\def\thelemma{\ref{lem:q3nf}}
\begin{lemma}
Over the field $\F_{7^2}$, the matrix $\mathbf{A^*_{\rho}}$ is a $2$-linear representation of a matroid which contains both $Q_3(\mathsf{Q}_8)$ and the non-Fano matroid\remove{ $F^-_7$} as minors.
\end{lemma}
\addtocounter{lemma}{-1}
\endgroup
\begin{lemma}[Lemma \ref{lem:q3nf}]
Over the field $\F_{7^2}$, the matrix $\mathbf{A^*_{\rho}}$ is a $2$-linear representation of a matroid which contains both $Q_3(\mathsf{Q}_8)$ and the non-Fano matroid\remove{ $F^-_7$} as minors.
\end{lemma}}
In order to verify that this is indeed a 2-linear representation, we need to verify that the rank function of the matroid is properly defined. This happens if and only if the rank of every relevant sub-matrix of $\mathbf{A^*_{\rho}}$ is even.
Since a proof by computer calculation would give little insight, we give a more constructive proof. We will see when adding a block column of $I_k$s results in a $k$-linear representation of a new matroid, which will show that $\mathbf{A^*_{\rho}}$ is not a $2$-linear representation over fields with characteristics 3 or 5.
\begin{proof}[Proof of Lemma \ref{lem:q3nf}]
As $\rho$ is fixed-point free, by Theorem \ref{th:dowlingmul} the rank function is properly defined for the restriction to $\set{1,2,3,\mathsf{e}^{(1)},...,(-\mathsf{k})^{(3)}}$, and that restriction is $Q_3(\mathsf{Q}_8)$. So we are left with checking subsets which contain $O$. Ranks of some subsets are easily shown to be integers, e.g., for any $g\in\mathsf{Q}_8$ we have
$$
\rank\left[\begin{matrix}
I_2&0&I_2\\
0&I_2&I_2\\
0&0&I_2
\end{matrix}
\right] =
\rank \left[\begin{matrix}
I_2&-I_2&I_2\\
0&\rho(g)&I_2\\
0&0&I_2
\end{matrix}
\right]=6.
$$
So $r(\set{1,2,O})=r(\set{1,g^{(1)},O})=3\in \N$. 
Also, for $g_i\neq g_j$ in $\mathsf{Q_8}$, we know from Theorem \ref{th:dowlingmul} that $\rank \left[\begin{matrix}
-I_2&-I_2\\
\rho(g_i)&\rho(g_j)\\
0&0
\end{matrix}
\right]=4,$ because $r(\set{g_i^{(1)},g_j^{(1)}})=2$ in the Dowling geometry. It follows that $\rank \left[\begin{matrix}
-I_2&-I_2&I_2\\
\rho(g_i)&\rho(g_j)&I_2\\
0&0&I_2
\end{matrix}
\right]=6.$ Thus, $r(\set{g_i^{(1)},g_j^{(1)},O})=3\in \N$ in $Q^{NF}_3(\mathsf{Q}_8)$. 

It remains to verify that the ranks of 
\begin{eqnarray*}
I)\left[\begin{matrix}
-I_2&0&I_2\\
\rho(g_i)&-I_2&I_2\\
0&\rho(g_j)&I_2
\end{matrix}
\right]{\raisebox{-3ex}{,}} & & II)
\left[\begin{matrix}
0&-I_2&I_2\\
0&\rho(g)&I_2\\
I_2&0&I_2
\end{matrix}
\right]{\raisebox{-3ex}{,}}
\end{eqnarray*}
are even, since the rest follows by permutation of the row/column blocks.
\remove{Ranks of other relevant sub-matrices follow from symmetrical arguments.}

For case II we see that
\begin{eqnarray*}\rank
\left[\begin{matrix}
0&-I_2&I_2\\
0&\rho(g)&I_2\\
I_2&0&I_2
\end{matrix}
\right]=2+\rank(\rho(\mathsf{e})-\rho(-g)),
\end{eqnarray*}
where $-g$ denotes $-\mathsf{e}\cdot g$. As $\rho$ is fixed-point free, we see, by Observation \ref{ob:subtraction}, that 
\begin{equation}
\rank(\rho(\mathsf{e})-\rho(-g))=\left\lbrace\begin{matrix}
0&g=-\mathsf{e}\\
2&{\rm otherwise.}\end{matrix}\right.
\end{equation}

For case I, we get that
\begin{eqnarray*}\rank
\left[\begin{matrix}
-I_2&0&I_2\\
\rho(g_i)&-I_2&I_2\\
0&\rho(g_j)&I_2
\end{matrix}
\right]=&4+\rank(I_2+\rho(g_j)+\rho(g_jg_i)).
\end{eqnarray*}
Since
$\rho(g_j),\rho(g_jg_i)\in \set{\pm \left(\begin{matrix}
1&0\\
0&1
\end{matrix}
\right),\pm \left(\begin{matrix}
\zeta_4&0\\
0&-\zeta_4
\end{matrix}
\right),\pm \left(\begin{matrix}
0&1\\
-1&0
\end{matrix}
\right),\pm \left(\begin{matrix}
0&\zeta_4\\
\zeta_4&0
\end{matrix}
\right)},$
it is easy to see that  $I_2+\rho(g_j)+\rho(g_jg_i)=\left(\begin{matrix}
a+b\zeta_4&c+d\zeta_4\\
-c+d\zeta_4&a-b\zeta_4
\end{matrix}
\right)$ with $a,b,c,d\in\N$ and $(|a|+|b|+|c|+|d|)\in \set{1,3}$ (it is $1$ if $g_j=-\mathsf{e}$, $g_i=-\mathsf{e}$, or $g_jg_i=-\mathsf{e}$, and 3 otherwise). Therefore, \begin{center}$\det(I_2+\rho(g_j)+\rho(g_jg_i))=a^2+b^2+c^2+d^2\neq 0 \mod\ 7$.\end{center} So $\rank(I_2+\rho(g_j)+\rho(g_jg_i))=2,$ as desired.

\begin{remark}
Notice that if we tried to define the same multi-linear representation $\mathbf{A^*_{\rho}}$ in characteristic 3 (by defining $\rho\colon\mathsf{Q}_8\to \F_{9}$ in a similar way), the assignment $a=b=c=1$ would give us $\rank(I_2+\rho(g_j)+\rho(g_jg_i))=1$. In characteristic 5 this would happen for, e.g., $a=2,b=1$. Therefore, this proof does not work in these characteristics. Indeed, in these characteristics the matrix $\mathbf{A^*_{\rho}}$ is not a $2$-linear representation of a matroid.
\end{remark}

Now the restriction to $\set{p_1,p_2,p_3,\mathsf{e}^{(1)},...,(-\mathsf{k})^{(3)}}$, by Theorem \ref{th:dowlingmul}, is $Q_3(\mathsf{Q}_8)$. And direct calculation shows that the restriction to $\set{p_1,p_2,p_3,(\mathsf{-e})^{(1)},(\mathsf{-e})^{(2)},(\mathsf{-e})^{(3)},O}$ is the non-Fano matroid, e.g.,
\begin{eqnarray*}
r(\set{p_1,(\mathsf{-e})^{(2)},O}) &= \rank\left[\begin{matrix}
I_2&0&I_2\\
0&-I_2&I_2\\
0&-I_2&I_2
\end{matrix}
\right]/2 &= 2,\\
r(\set{(\mathsf{-e})^{(1)},(\mathsf{-e})^{(2)},(\mathsf{-e})^{(3)}}) &= \rank\left[\begin{matrix}
-I_2&0&-I_2\\
-I_2&-I_2&0\\
0&-I_2&-I_2
\end{matrix}
\right]/2 &= 3.
\end{eqnarray*}
\end{proof}

\subsection*{Acknowledgements}
I would like to thank my advisors Amos Beimel and Ilya Tyomkin for drawing my attention to this problem, for many helpful discussions, and for the guidance and assistance in the preparation of this paper.

\bibliography{rbib}
\bibliographystyle{plain}

\end{document}